\documentclass[reqno]{amsart}

\usepackage[T1]{fontenc}

\usepackage{amsmath}
\usepackage{amssymb}
\usepackage{amsthm}
\usepackage{amsfonts}

\usepackage{mathtools}
\usepackage{enumerate}
\usepackage[hidelinks]{hyperref}
\usepackage[nameinlink]{cleveref}
\usepackage{amsrefs}

\numberwithin{equation}{section}

\theoremstyle{plain}
\newtheorem{theorem}{Theorem}
\newtheorem{proposition}{Proposition}[section]
\newtheorem{lemma}[proposition]{Lemma}

\theoremstyle{definition}
\newtheorem{definition}{Definition}[section]

\newcommand{\R}{\mathbb{R}}

\newcommand{\cH}{\mathcal{H}}
\newcommand{\cM}{\mathcal{M}}

\newcommand\cl[1]{\overline{#1}}
\newcommand{\loc}{\mathrm{loc}}
\newcommand{\weakto}{\rightharpoonup}

\newcommand{\defeq}{\vcentcolon=}

\renewcommand{\d}{\mathop{}\!\mathrm{d}}

\DeclarePairedDelimiter{\abs}{\lvert}{\rvert}

\DeclarePairedDelimiter{\norm}{\lVert}{\rVert}
\DeclarePairedDelimiter{\paren}{\lparen}{\rparen}
\DeclarePairedDelimiter{\set}{\lbrace}{\rbrace}

\DeclareMathOperator{\supp}{supp}

\begin{document}

\title[Representation formula for the normal derivative]{A representation formula for the distributional normal derivative}

\author{Augusto C. Ponce}
\address{
Augusto C. Ponce\hfill\break\indent
Universit{\'e} catholique de Louvain\hfill\break\indent
Institut de Recherche en Math{\'e}matique et Physique\hfill\break\indent
Chemin du Cyclotron 2, bte L7.01.02\hfill\break\indent
1348 Louvain-la-Neuve\hfill\break\indent
Belgium}
\email{Augusto.Ponce@uclouvain.be}

\author{Nicolas Wilmet}
\address{
Nicolas Wilmet\hfill\break\indent
Universit{\'e} catholique de Louvain\hfill\break\indent
Institut de Recherche en Math{\'e}matique et Physique\hfill\break\indent
Chemin du Cyclotron 2, bte L7.01.02\hfill\break\indent
1348 Louvain-la-Neuve\hfill\break\indent
Belgium}
\email{Nicolas.Wilmet@uclouvain.be}


\begin{abstract}
We prove an integral representation formula for the distributional normal derivative of solutions of
\[
\begin{dcases}
\begin{aligned}
- \Delta u + V u &= \mu && \text{in \(\Omega\),} \\
u &= 0 && \text{on \(\partial\Omega\),}
\end{aligned}
\end{dcases}
\]
where \(V \in L_\loc^1(\Omega)\) is a nonnegative function and \(\mu\) is a finite Borel measure on \(\Omega\). As an application, we show that the Hopf lemma holds almost everywhere on \(\partial\Omega\) when \(V\) is a nonnegative Hopf potential.
\end{abstract}

\maketitle

\section{Introduction and main results}

In \cite{BrezisPonce:2008}, H.~Brezis and the first author introduced a notion of distributional normal derivative that applies in particular to solutions of the Dirichlet problem
\begin{equation}
\label{eq:dp}
\begin{dcases}
\begin{aligned}
- \Delta u + V u &= \mu && \text{in \(\Omega\),} \\
u &= 0 && \text{on \(\partial\Omega\),}
\end{aligned}
\end{dcases}
\end{equation}
where \(\Omega \subset \R^N\) is a smooth connected bounded open set, \(V \in L_\loc^1(\Omega)\) is a nonnegative function and \(\mu\) is a finite Borel measure on \(\Omega\). By a solution of \eqref{eq:dp}, we mean a function \(u \in W_0^{1, 1}(\Omega) \cap L^1(\Omega; V \d{x})\) such that
\[
- \Delta u + V u = \mu \quad \text{in the sense of distributions in \(\Omega\).}
\]
For example, when \(\mu = f \d{x}\) with \(f \in L^2(\Omega)\) the solution always exists and can be obtained as the minimizer of the functional
\[
E(z) = \frac{1}{2} \int_\Omega {} (\abs{\nabla z}^2 + V z^2) \d{x} - \int_\Omega f z \d{x}
\]
in \(W_0^{1, 2}(\Omega) \cap L^2(\Omega; V \d{x})\)\,; see \cite{DalMasoMosco:1986}*{Theorem~2.4}. The assumption \(V \in L_\loc^1(\Omega)\) ensures that \(C_c^\infty(\Omega) \subset L^2(\Omega; V \d{x})\) and thus smooth functions with compact support are admissible in the Euler--Lagrange equation associated to \(E\).

The distributional normal derivative of \(u\) is an element of \(L^1(\partial\Omega)\), which we denote by \(\partial u / \partial n\) and coincides with the classical normal derivative when \(u \in C^2(\cl\Omega)\), that is characterized by the identity
\begin{equation}
\label{eq:integral_identity_distributional_normal_derivative}
\int_\Omega \nabla u \cdot \nabla \phi \d{x} = - \int_\Omega \phi \Delta u - \int_{\partial\Omega} \frac{\partial u}{\partial n} \phi \d\sigma \quad \text{for every \(\phi \in C^\infty(\cl\Omega)\),}
\end{equation}
where \(\sigma = \cH^{N - 1}\lfloor_{\partial\Omega}\) denotes the surface measure on \(\partial\Omega\)\,; see \cite{BrezisPonce:2008}*{Theorem~1.2} or \cite{Ponce:2016}*{Proposition~7.3}. We adopt the convention that \(n\) is the \emph{inward} unit normal vector on \(\partial\Omega\), which explains the minus sign in front of the second integral in the right-hand side of \eqref{eq:integral_identity_distributional_normal_derivative}.

Our goal in this paper is to prove an integral representation formula for \(\partial u / \partial n\) that involves \(\mu\) and is valid for all solutions of \eqref{eq:dp}. To obtain the kernel of such a formula, we rely upon a counterpart of the notion of duality solution by A.~Malusa and L.~Orsina~\cite{MalusaOrsina:1996}, inspired from the seminal paper \cite{LittmanStampacchiaWeinberger:1963}. We adapt this formalism to the boundary value problem
\begin{equation}
\label{eq:bvp_dirac}
\begin{dcases}
\begin{aligned}
- \Delta v + V v &= 0 && \text{in \(\Omega\),} \\
v &= \delta_a && \text{on \(\partial\Omega\),}
\end{aligned}
\end{dcases}
\end{equation}
involving the Dirac measure \(\delta_a\) on \(\partial\Omega\).

\begin{definition}
Given \(a \in \partial\Omega\), a function \(P_a \in L^1(\Omega)\) is a duality solution of \eqref{eq:bvp_dirac} whenever
\[
\frac{\widehat{\partial \zeta_f}}{\partial n}(a) = \int_\Omega P_a f \d{x} \quad \text{for every \(f \in L^\infty(\Omega)\),}
\]
where \(\zeta_f\) denotes the solution of \eqref{eq:dp} with datum \(\mu = f \d{x}\).
\end{definition}

In the identity above, \(\widehat{\partial \zeta_f} / \partial n\) is a pointwise representative of \(\partial \zeta_f / \partial n\) that we introduce in \cite{PonceWilmet:2020} as follows. One first considers the solution of
\begin{equation}
\label{eq:approx_dp}
\begin{dcases}
\begin{aligned}
- \Delta \zeta_{f, k} + \min \set{V, k} \zeta_{f, k} &= f && \text{in \(\Omega\),} \\
\zeta_{f, k} &= 0 && \text{on \(\partial\Omega\),}
\end{aligned}
\end{dcases}
\end{equation}
which belongs to \(C^1(\cl\Omega)\) and converges to \(\zeta_f\) in \(L^1(\Omega)\) as \(k \to \infty\)\,; see \cite{OrsinaPonce:2020}*{Lemma~3.2}. We prove in \cite{PonceWilmet:2020} that \((\partial \zeta_{f, k} / \partial n)\) converges pointwise on \(\partial\Omega\) and the function
\[
\frac{\widehat{\partial \zeta_f}}{\partial n}(a) \defeq \lim_{k \to \infty} \frac{\partial \zeta_{f, k}}{\partial n}(a) \quad \text{for \(a \in \partial\Omega\)}
\]
is the distributional normal derivative of \(\zeta_f\)\,; see \cite{PonceWilmet:2020}*{Proposition~2.1}.

Existence and uniqueness of the duality solution \(P_a\) for every \(a \in \partial\Omega\) is a straightforward matter that is easily handled using the Riesz representation theorem; see~\cite{PonceWilmet:2020}*{Proposition~3.1}. The duality solution \(P_a\) can also be obtained from an approximation scheme based on a truncation of \(V\) (analogous to \eqref{eq:approx_dp}); see Lemma~\ref{lem:measure_limit} below. Then, by Fatou's lemma,
\begin{equation}
\label{eq:Pa_subsolution}
- \Delta P_a + V P_a \le 0 \quad \text{in the sense of distributions in \(\Omega\).}
\end{equation}
As \(P_a\) is nonnegative, \(P_a\) is subharmonic, whence by local boundedness of \(P_{a}\) its precise representative \(\widehat{P_a}\) exists everywhere on \(\Omega\)\,; see, e.g., \cite{Ponce:2016}*{Lemma~8.10}. The latter means that
\[
\lim_{r \to 0} \frac{1}{r^N} \int_{B_r(x)} {} \abs[\big]{P_a - \widehat{P_a}(x)} \d{y} = 0
\quad
\text{for every \(x \in \Omega\).}
\]
Using this family of functions \(\widehat{P_a}\), we prove a representation formula for the distributional normal derivative \(\partial u / \partial n\) that is valid almost everywhere on \(\partial\Omega\)\,:

\begin{theorem}
\label{thm:representation_formula}
If \(u\) satisfies \eqref{eq:dp} for some finite Borel measure \(\mu\) on \(\Omega\), then, for almost every \(a \in \partial\Omega\), we have \(\widehat{P_a} \in L^1(\Omega; \abs{\mu})\) and
\[
\frac{\partial u}{\partial n}(a) = \int_\Omega \widehat{P_a} \d\mu.
\]
\end{theorem}

We prove in \cite{PonceWilmet:2020} that, for every \(a \in \partial\Omega\),
\begin{enumerate}[(i)]
\item either \(P_a = 0\) almost everywhere on \(\Omega\),
\item or \(P_a\) is a distributional solution of \eqref{eq:bvp} in the sense that \(VP_a \in L^1(\Omega; d_{\partial\Omega} \d{x})\) and
\[
\frac{\partial \zeta}{\partial n}(a) = \int_\Omega P_a (- \Delta \zeta + V \zeta) \d{x} \quad \text{for every \(\zeta \in C_0^\infty(\cl\Omega)\),}
\]
where \(d_{\partial\Omega} : \Omega \to \R_+\) is the distance to the boundary and \(C_0^\infty(\cl\Omega)\) is the set of functions \(\zeta \in C^\infty(\cl\Omega)\) such that \(\zeta = 0\) on \(\partial\Omega\).
\end{enumerate}
In particular, \eqref{eq:Pa_subsolution} always holds with equality. Using this fact and \Cref{thm:representation_formula}, we give a new proof of the Hopf boundary point lemma from \cite{OrsinaPonce:2018}.

The main issue concerning the validity of the Hopf lemma for the Schr\"odinger operator is that the growth of \(V\) near the boundary may be so strong that the requirements that
\begin{equation}
\label{eq:requirements}
V u \in L^1(\Omega) \quad \text{and} \quad \frac{\partial u}{\partial n} > 0 \text{ almost everywhere on \(\partial\Omega\)}
\end{equation}
are incompatible. Such is the case when \(V = 1 / d_{\partial\Omega}^2\) as one cannot have both properties in \eqref{eq:requirements} when \(u \in C^\infty(\cl\Omega)\) (and even for nonsmooth solutions). The notion of Hopf potential rules out that type of obstruction. We first recall

\begin{definition}
\label{def:hopf_potential}
We say that \(V \in L_\loc^1(\Omega)\) is a Hopf potential whenever there exists a nonnegative function \(\theta \in W_0^{1, 1}(\Omega)\) such that \(\Delta \theta\) is a finite Borel measure on \(\Omega\), \(V \theta \in L^1(\Omega)\) and \(\partial \theta / \partial n > 0\) almost everywhere on \(\partial\Omega\).
\end{definition}

A more general definition is presented in \cite{OrsinaPonce:2018} but one shows that they are equivalent; see \cite{OrsinaPonce:2018}*{Proposition~2.5}. As an example, every potential \(V \in L^1(\Omega; d_{\partial\Omega} \d{x})\) is a Hopf potential. Indeed, one can take as \(\theta\) in \Cref{def:hopf_potential} any nonnegative function \(\zeta \in C_0^\infty(\cl\Omega)\) such that \(\partial \zeta / \partial n > 0\) on \(\partial\Omega\).

\begin{theorem}
\label{thm:hopf_lemma}
If\/ \(V\) is a nonnegative Hopf potential and if \(u\) is the solution of \eqref{eq:dp} involving a nonnegative finite Borel measure \(\mu\) on \(\Omega\), \(\mu \neq 0\), then
\[
\frac{\partial u}{\partial n} > 0 \quad \text{almost everywhere on \(\partial\Omega\).}
\]
\end{theorem}

\section{Proof of \Cref{thm:representation_formula}}

We denote by \(\cM(\Omega)\) the vector space of finite Borel measures on \(\Omega\), equipped with the norm
\[
\norm{\mu}_{\cM(\Omega)} = \abs{\mu}(\Omega).
\]
We recall that if \(u\) is the solution of \eqref{eq:dp} involving \(\mu \in \cM(\Omega)\), then for all \(1 \le p < \frac{N}{N - 1}\), we have \(u \in W_0^{1, p}(\Omega)\) and
\begin{equation}
\label{eq:W1p_estimate}
\norm{u}_{W^{1, p}(\Omega)} \le C \norm{\mu}_{\cM(\Omega)}
\end{equation}
for some constant \(C > 0\) depending on \(p\) and \(\Omega\). This can be obtained from elliptic estimates due to Littman, Stampacchia and Weinberger~\cite{LittmanStampacchiaWeinberger:1963}*{Theorem~5.1} and from the absorption estimate
\begin{equation}
\label{eq:absorption_estimate}
\norm{V u}_{L^1(\Omega)} \le \norm{\mu}_{\cM(\Omega)}\,;
\end{equation}
see~\cite{BrezisMarcusPonce:2007}*{Proposition~4.B.3} or \cite{Ponce:2016}*{Proposition~21.5}.

We also recall that if \(u \in W_0^{1, 1}(\Omega)\) is such that \(\Delta u \in \cM(\Omega)\), then the distributional normal derivative of \(u\) satisfies
\[
\norm[\bigg]{\frac{\partial u}{\partial n}}_{L^1(\partial\Omega)} \le \norm{\Delta u}_{\cM(\Omega)}\,;
\]
see \cite{BrezisPonce:2008}*{Theorem~1.2} or \cite{Ponce:2016}*{Proposition~7.3}. In the case where \(u\) is the solution of \eqref{eq:dp} involving \(\mu \in \cM(\Omega)\), one deduces from the above inequality and the absorption estimate \eqref{eq:absorption_estimate} that
\begin{equation}
\label{eq:L1_estimate_distributional_normal_derivative}
\norm[\bigg]{\frac{\partial u}{\partial n}}_{L^1(\partial\Omega)} \le 2 \norm{\mu}_{\cM(\Omega)}.
\end{equation}
We use this estimate to deduce the following property:

\begin{lemma}
\label{lem:fatou_measure}
For every nonnegative measure \(\mu \in \cM(\Omega)\), we have
\[
\widehat{P_a} \in L^1(\Omega; \mu) \quad \text{for almost every \(a \in \partial\Omega\).}
\]
\end{lemma}

\begin{proof}
Let \(u_k\) be the solution of
\begin{equation}
\begin{dcases}
\begin{aligned}
- \Delta u_k + V u_k &= \rho_k * \mu && \text{in \(\Omega\),} \\
u_k &= 0 && \text{on \(\partial\Omega\),}
\end{aligned}
\end{dcases}
\end{equation}
where \((\rho_{k})\) is a sequence of radial mollifiers in \(C_{c}^{\infty}(\R^{N})\) such that \(\supp \rho_k \subset B_{1/k}(0)\) and \(\rho_k * \mu : \R^N \to \R\) is the smooth function defined for every \(x \in \R^N\) by
\[
\rho_k * \mu(x) = \int_\Omega \rho_k(x - y) \d\mu(y).
\]
Then, by definition of \(P_a\) and Fubini's theorem,
\[
\frac{\widehat{\partial u_k}}{\partial n}(a) = \int_\Omega P_a \, \rho_k * \mu \d{x} = \int_\Omega \rho_k * P_a \d\mu \quad \text{for every \(a \in \partial\Omega\).}
\]
Since
\[
\frac{\partial u_k}{\partial n} = \frac{\widehat{\partial u_k}}{\partial n} \quad \text{almost everywhere on \(\partial\Omega\),}
\]
we deduce from \eqref{eq:L1_estimate_distributional_normal_derivative} that
\[
\norm[\bigg]{\widehat{\frac{\partial u_k}{\partial n}}}_{L^1(\partial\Omega)} \le 2 \norm{\rho_k * \mu}_{L^1(\Omega)} \le 2 \norm{\mu}_{\cM(\Omega)}.
\]
Since \(P_a\) and \(\mu\) are nonnegative, we thus have
\[
\int_{\partial\Omega} \paren[\bigg]{\int_\Omega \rho_k * P_a \d\mu} \d\sigma(a) \le 2 \norm{\mu}_{\cM(\Omega)}.
\]
On the other hand, since every point of \(\Omega\) is a Lebesgue point of \(P_a\),
\[
\rho_k * P_a \to \widehat{P_a} \quad \text{everywhere on \(\Omega\)\,;}
\]
see \cite{Ponce:2016}*{Proposition~8.4}. It then follows from Fatou's lemma that
\[
\int_{\partial\Omega} \paren[\bigg]{\int_\Omega \widehat{P_a} \d\mu} \d\sigma(a) \le \liminf_{k \to \infty} \int_{\partial\Omega} \paren[\bigg]{\int_\Omega \rho_k * P_a \d\mu} \d\sigma(a) \le 2 \norm{\mu}_{\cM(\Omega)}.
\]
In particular,
\[
\int_\Omega \widehat{P_a} \d\mu < + \infty \quad \text{for almost every \(a \in \partial\Omega\).}\qedhere
\]
\end{proof}

We now prove that the precise representative \(\widehat{P_{a}}\) can also be recovered from an approximation scheme involving the Schr\"odinger operator with truncated potentials.

\begin{lemma}
\label{lem:measure_limit}
For every \(a \in \partial\Omega\), denote by \(P_{a, k}\) the distributional solution of
\[
\begin{dcases}
\begin{aligned}
- \Delta P_{a, k} + V_k P_{a, k} &= 0 && \text{in \(\Omega\),} \\
P_{a, k} &= \delta_a && \text{on \(\partial\Omega\),}
\end{aligned}
\end{dcases}
\]
where \((V_{k})\) is a nondecreasing sequence of nonnegative functions in \(L^{\infty}(\Omega)\) that converges almost everywhere to \(V\) on \(\Omega\). Then, for every \(a \in \partial\Omega\), the sequence \((\widehat{P_{a, k}})\) is non-increasing and converges everywhere to \(\widehat{P_a}\) on \(\Omega\).
\end{lemma}

\begin{proof}
Let \(a \in \partial\Omega\). One shows that the sequence \((\widehat{P_{a, k}})\) is nonnegative and non-increasing; this follows from a straightforward counterpart of the weak maximum principle \cite{PonceWilmet:2020}*{Lemma~2.2} for distributional solutions of \eqref{eq:bvp} and from the comparison principle \cite{PonceWilmet:2020}*{Lemma~2.3}. We denote by \(z_k\) the solution of
\[
\begin{dcases}
\begin{aligned}
- \Delta z_k &= V_k P_{a, k} && \text{in \(\Omega\),} \\
z_k &= 0 && \text{on \(\partial\Omega\),}
\end{aligned}
\end{dcases}
\]
and by \(K_a\) the solution of
\begin{equation}
\label{eq:Ka_dp}
\begin{dcases}
\begin{aligned}
- \Delta K_a &= 0 && \text{in \(\Omega\),} \\
K_a &= \delta_a && \text{on \(\partial\Omega\).}
\end{aligned}
\end{dcases}
\end{equation}
Then by uniqueness of solutions \(z_k = K_a - P_{a, k}\) almost everywhere on \(\Omega\), and thus
\[
\widehat{z_k} = \widehat{K_a} - \widehat{P_{a, k}} \quad \text{on \(\Omega\).}
\]
Hence \((\widehat{z_k})\) is a nondecreasing and locally uniformly bounded sequence of nonnegative superharmonic functions. Its pointwise limit \(z\) thus coincides with \(\hat{z}\) on \(\Omega\)\,; see \cite{MalusaOrsina:1996}*{Lemma~4.12}. Therefore, we have
\[
\hat{z} = \widehat{K_a} - \lim_{k \to \infty} \widehat{P_{a, k}} \quad \text{on \(\Omega\).}
\]
Since \(z = K_a - P_a\) almost everywhere on \(\Omega\), and \(z\) is a locally bounded superharmonic function, we obtain \(\widehat{P_a} = \lim\limits_{k \to \infty} \widehat{P_{a, k}}\) as claimed.
\end{proof}

As a preliminary step in the proof of Theorem~\ref{thm:representation_formula}, we begin with the following case:

\begin{lemma}
\label{lem:measure_compact}
Theorem~\ref{thm:representation_formula} holds when \(V\) is bounded and \(\supp \mu\) is a compact subset of \(\Omega\).
\end{lemma}

\begin{proof}
Let \((\rho_k)\) be a sequence of radial mollifiers in \(C_c^\infty(\R^N)\) such that \(\supp \rho_k \subset B_{1/k}(0)\) and denote by \(u_k\) the solution of \eqref{eq:dp} with datum \(\rho_k * \mu \d x\). We first show that
\begin{equation}
\label{eq:1264}
\lim_{k \to \infty}{\int_{\partial\Omega} \frac{\partial u_k}{\partial n} \psi \d\sigma} = \int_{\partial\Omega} \frac{\partial u}{\partial n} \psi \d\sigma \quad \text{for every \(\psi \in C^0(\partial\Omega)\).}
\end{equation}
To this end, we recall that from \eqref{eq:W1p_estimate} we have, for every \(1 < p < \frac{N}{N - 1}\),
\[
\norm{u_k}_{W^{1, p}(\Omega)} \le C \norm{\rho_k * \mu}_{L^1(\Omega)} \le C \norm{\mu}_{\cM(\Omega)}.
\]
Therefore, passing to a subsequence if necessary, we may assume that
\[
u_k \weakto v 
\quad \text{and} \quad
\nabla u_k \weakto \nabla v
\quad
\text{weakly in \(L^{1}(\Omega)\)}
\]
for some function \(v \in W^{1, 1}_{0}(\Omega)\). On the other hand, we have
\[
\int_\Omega \varphi \rho_k * \mu \d{x} \to \int_\Omega \varphi \d\mu \quad \text{for every \(\varphi \in C_c^0(\Omega)\)\,;}
\]
see, e.g., \cite{Ponce:2016}*{Proposition~2.7}. By boundedness of \(V\), it thus follows that \(v\) is also a solution of \eqref{eq:dp} involving \(\mu\). Then, by uniqueness of solutions, \(v = u\) almost everywhere on \(\Omega\). We now recall that, for every \(\phi \in C^\infty(\cl\Omega)\),
\begin{equation}
\label{eq:1389}
\int_{\partial\Omega} \frac{\partial u_k}{\partial n} \phi \d\sigma = - \int_\Omega \nabla u_k \cdot \nabla \phi \d{x} - \int_\Omega V u_k \phi \d{x} + \int_\Omega \rho_k * \mu \, \phi \d{x}.
\end{equation}
Since \((\rho_k * \phi)\) is uniformly bounded and converges locally uniformly to \(\phi\) on \(\Omega\), using Fubini's theorem we get
\[
\lim_{k \to \infty}{\int_\Omega \rho_k * \mu \, \phi \d{x}} = \lim_{k \to \infty}{\int_\Omega \rho_k * \phi \d\mu} = \int_\Omega \phi \d\mu.
\]
As \(k \to \infty\) in \eqref{eq:1389}, we deduce \eqref{eq:1264} using the boundedness of \(V\).

We now prove that
\begin{equation}
\label{eq:1298}
\lim_{k \to \infty} \int_{\partial\Omega} \frac{\partial u_k}{\partial n} \psi \d\sigma = \int_{\partial\Omega} \paren[\bigg]{\int_\Omega \widehat{P_a} \d\mu} \psi(a)  \d\sigma(a) \quad \text{for every \(\psi \in C^0(\partial\Omega)\).}
\end{equation}
By boundedness of \(V\), we have \(u_k \in C^1(\cl\Omega)\) and then
\[
\frac{\partial u_k}{\partial n} = \frac{\widehat{\partial u_k}}{\partial n} \quad \text{on \(\partial\Omega\).}
\]
Fubini's theorem implies that, for every \(a \in \partial\Omega\),
\begin{equation}
\label{eq:1413}
\frac{\partial u_k}{\partial n}(a) = \int_\Omega P_a \, \rho_k * \mu \d{x} = \int_\Omega \rho_k * P_a \d\mu.
\end{equation}
Since \(V\) is bounded, \(\widehat{P_{a}}\) is continuous on \(\Omega\) and then
\[
\lim_{k \to \infty}{\rho_k * P_a(x)} = \widehat{P_{a}}(x) \quad \text{for every \(x \in \Omega\).}
\]
As this convergence is locally uniform on \(\Omega\) and \(\supp{\mu}\) is a compact subset of \(\Omega\), we get
\begin{equation}
\label{eq:1425}
\lim_{k \to \infty} \frac{\partial u_k}{\partial n}(a) = \int_\Omega \widehat{P_{a}} \d\mu \quad \text{for every \(a \in \partial\Omega\).}
\end{equation}

Let \(a \in \partial\Omega\). By comparison, \(0 \le P_a \le K_a\) almost everywhere on \(\Omega\), where \(K_a\) is the solution of \eqref{eq:Ka_dp}. Then
\[
0 \le \rho_k * P_a \le \rho_k * K_a \quad \text{on \(\Omega\).}
\]
Since \(K_a\) is harmonic on \(\Omega\), \(\supp \mu\) is a compact subset of \(\Omega\) and \(\rho_k\) is radial, we also have, for \(k\) large enough,
\[
\rho_k * K_a \le K_a \le C \quad \text{almost everywhere on \(\supp \mu\)}
\]
for some constant \(C > 0\) independent of \(a\). Then, from \eqref{eq:1413}, we get
\[
\abs[\bigg]{\frac{\partial u_k}{\partial n}(a)} \le \int_\Omega C \d\abs{\mu} = C \norm{\mu}_{\cM(\Omega)}.
\]
Thus, the sequence \((\partial u_k / \partial n)\) is uniformly bounded on \(\partial\Omega\). Using the pointwise convergence \eqref{eq:1425} and Lebesgue's dominated convergence theorem, we obtain \eqref{eq:1298}.

Combining \eqref{eq:1264} and \eqref{eq:1298}, we get
\[
\int_{\partial\Omega} \frac{\partial u}{\partial n} \psi \d\sigma = \int_{\partial\Omega} \paren[\bigg]{\int_\Omega \widehat{P_a} \d\mu} \psi(a)  \d\sigma(a) \quad \text{for every \(\psi \in C^0(\partial\Omega)\).}
\]
Therefore
\[
\frac{\partial u}{\partial n}(a) = \int_\Omega \widehat{P_a} \d\mu  \quad \text{for almost every \(a \in \partial\Omega\).}
\qedhere
\]
\end{proof}

We now proceed with the

\begin{proof}[Proof of \Cref{thm:representation_formula}]
Let \((\omega_k)\) be a nondecreasing sequence of non-empty open subsets \(\omega_k \Subset \Omega\) such that
\[
\Omega = \bigcup_{k = 1}^\infty \omega_k.
\]
Set \(\mu_k = \mu\lfloor_{\omega_k}\) and denote by \(w_k\) the solution of
\[
\begin{dcases}
\begin{aligned}
- \Delta w_k + V_{k} w_k &= \mu_k && \text{in \(\Omega\),} \\
u_k &= 0 && \text{on \(\partial\Omega\),}
\end{aligned}
\end{dcases}
\]
where \((V_{k})\) is a sequence as in \Cref{lem:measure_limit}. Observe that \(w_{k} - u\) satisfies the Dirichlet problem
\[
\begin{dcases}
\begin{aligned}
- \Delta (w_k - u) + V_{k} (w_k - u) &= (\mu_k - \mu) + (V - V_{k})u && \text{in \(\Omega\),} \\
w_k - u &= 0 && \text{on \(\partial\Omega\).}
\end{aligned}
\end{dcases}
\]
Thus, by \eqref{eq:L1_estimate_distributional_normal_derivative} and the triangle inequality, we have
\[
\norm[\bigg]{\frac{\partial w_{k}}{\partial n} - \frac{\partial u}{\partial n}}_{L^{1}(\partial\Omega)}
\le 2 \bigl(\norm{\mu_{k} - \mu}_{\cM(\Omega)} + \norm{(V - V_{k})u}_{L^{1}(\Omega)} \bigr),
\]
which implies that
\[
\frac{\partial w_{k}}{\partial n}
\to \frac{\partial u}{\partial n}
\quad \text{in \(L^{1}(\partial\Omega)\).}
\]
Therefore, taking a subsequence if necessary, we may assume that
\[
\frac{\partial w_k}{\partial n} \to \frac{\partial u}{\partial n} \quad \text{almost everywhere on \(\partial\Omega\).}
\]

To conclude, it suffices to show that
\begin{equation}
\label{eq:1488}
\lim_{k \to \infty}{\frac{\partial w_{k}}{\partial n}(a)}
= \int_{\Omega}{\widehat{P_{a}} \d\mu}
\quad \text{for almost every \(a \in \partial\Omega\).}
\end{equation}
To this end, we observe that by \Cref{lem:measure_compact} we have
\[
\frac{\partial w_{k}}{\partial n}(a)
= \int_{\Omega}{\widehat{P_{a, k}} \d\mu_{k}}
= \int_{\Omega}{\widehat{P_{a, k}} \chi_{\omega_{k}}\d\mu}.
\]
Since \((P_{a, k})\) is nonnegative and non-increasing,
\[
0 \le \widehat{P_{a, k}} \chi_{\omega_{k}} \le \widehat{P_{a, 1}} \quad \text{on \(\Omega\).}
\]
By \Cref{lem:measure_limit}, the sequence \((\widehat{P_{a, k}} \chi_{\omega_{k}})\) converges everywhere to \(\widehat{P_{a}}\) on \(\Omega\) and, by \Cref{lem:fatou_measure} applied to \(\abs{\mu}\), \(\widehat{P_{a, 1}} \in L^{1}(\Omega; \abs{\mu})\) for almost every \(a \in \partial\Omega\). For any such a point \(a\), we can then apply Lebesgue's dominated convergence theorem to deduce \eqref{eq:1488}, which completes the proof.
\end{proof}

\section{Proof of \Cref{thm:hopf_lemma}}

The following lemma ensures that the set of points \(a \in \partial\Omega\) for which \(\widehat{P_a} \equiv 0\) on \(\Omega\) is negligible with respect to the \(\cH^{N - 1}\) Hausdorff measure.

\begin{lemma}
\label{lem:nonzero_hopf_potential}
If \(V\) is a nonnegative Hopf potential, then
\[
\widehat{P_a} \not\equiv 0 \quad \text{for almost every \(a \in \partial\Omega\).}
\]
\end{lemma}

\begin{proof}
Let \(\theta\) be the function that appears in \Cref{def:hopf_potential}. By \Cref{thm:representation_formula}, for almost every \(a \in \partial\Omega\) we have
\[
\int_\Omega \widehat{P_a} (- \Delta \theta + V \theta) = \frac{\partial \theta}{\partial n}(a) > 0.\qedhere
\]
\end{proof}

We denote by \(\Gamma\) the set of points \(a \in \partial\Omega\) such that \eqref{eq:bvp_dirac} has a distributional solution. For \(a \in \Gamma\), \(P_a\) is a nontrivial nonnegative supersolution of \(- \Delta + V\), and then the strong maximum principle for the Schr\"odinger operator  implies that
\[
P_a > 0 \quad \text{almost everywhere on \(\Omega\)\,;}
\]
see \cite{Ancona:1979}*{Th\'eor\`eme~9}, \cite{BrezisPonce:2003} or \cite{Trudinger:1978}.
The following lemma is more precise:

\begin{lemma}
\label{lem:zero-set}
For every \(a \in \Gamma\), we have
\[
\widehat{P_a} > 0 \quad \text{on \(G\),}
\]
where \(G\) is the set of points \(x \in \Omega\) such that \(\widehat{\zeta_f}(x) \neq 0\) for some \(f \in L^\infty(\Omega)\).
\end{lemma}

The proof of \Cref{lem:zero-set} relies upon a comparison principle for distributional solutions of the boundary value problem
\begin{equation}
\label{eq:bvp}
\begin{dcases}
\begin{aligned}
- \Delta v + V v &= 0 && \text{in \(\Omega\),} \\
v &= \nu && \text{on \(\partial\Omega\),}
\end{aligned}
\end{dcases}
\end{equation}
where \(\nu\) is a nonnegative finite Borel measure on \(\partial\Omega\), with some suitable solution of \eqref{eq:dp} with bounded datum. By a distributional solution \(v\) of \eqref{eq:bvp}, we mean that \(v \in L^1(\Omega)\) is such that \(V v \in L^1(\Omega; d_{\partial\Omega} \d{x})\) and
\[
\int_\Omega v (- \Delta \zeta + V \zeta) \d{x} = \int_{\partial\Omega} \frac{\partial \zeta}{\partial n} \d\nu \quad \text{for every \(\zeta \in C_0^\infty(\cl\Omega)\).}
\]

\begin{lemma}
\label{lem:comparison_principle}
There exists a bounded nondecreasing continuous function \(H : \R_+ \to \R_+\), with \(H(t) > 0\) for \(t > 0\), such that, for every nonnegative function \(V \in L_\loc^1(\Omega)\) and every nonnegative finite Borel measure \(\nu\) on \(\partial\Omega\), if \(v\) is the distributional solution of \eqref{eq:bvp} with datum \(\nu\), then \(v \ge \zeta_{H(v)}\) almost everywhere on \(\Omega\).
\end{lemma}

\Cref{lem:comparison_principle} extends Proposition~4.1 in \cite{OrsinaPonce:2018} which deals with the case where \(\nu = g \d\sigma\) with nonnegative \(g \in L^\infty(\partial\Omega)\). We take this result for granted, so that, in the proof of \Cref{lem:comparison_principle}, the function \(H\) is the one constructed in \cite{OrsinaPonce:2018}. In our approximation argument, we rely on the fact that \(H\) does not depend on the potential \(V\) nor on the datum \(\nu\). An admissible choice of \(H\) is
\[
H(t) =
\begin{cases}
	\epsilon t^{\alpha}
	& \text{for \(0 \le t \le 1\),}\\
	\epsilon
	& \text{for \(t > 1\),}	
\end{cases}
\]
where \(0 < \alpha < 1\) and \(\epsilon > 0\) is small and depends on \(\alpha\).

\begin{proof}[Proof of \Cref{lem:comparison_principle}]
We first assume that \(V\) is bounded. In this case, the boundary value problem \eqref{eq:bvp} has a unique solution for every datum \(\nu = g \d\sigma\) with \(g \in L^\infty(\partial\Omega)\). Let \((g_k)\) be a sequence of nonnegative smooth functions on \(\partial\Omega\) such that
\[
\int_{\partial\Omega} \psi g_k \d\sigma \to \int_{\partial\Omega} \psi \d\nu \quad \text{for every \(\psi \in C^0(\partial\Omega)\).}
\]
Such a sequence can be obtained, for instance, from a convolution of \(\nu\) with a sequence of mollifiers. We denote by \(v_k\) the solution of \eqref{eq:bvp} with datum \(g_k\). According to Proposition~4.1 in \cite{OrsinaPonce:2018}, there exists a function \(H : \R_{+} \to \R_{+}\) as in the statement above such that
\begin{equation}
\label{eqNew1}
v_k \ge \zeta_{H(v_k)} \quad \text{almost everywhere on \(\Omega\).}
\end{equation}
By standard estimates, we have
\[
\norm{v_k}_{W^{1, 1}(\omega)} \le C \norm{g_k}_{L^1(\partial\Omega)}
\]
for every \(\omega \Subset \Omega\) and some constant \(C > 0\) depending on \(\omega\)\,; see \cite{MarcusVeron:2014}*{Theorem~1.2.2 and Proposition~2.1.2}. Hence, we may assume that \((v_k)\) converges almost everywhere on \(\Omega\) to some function \(w\). Then, Lebesgue's dominated convergence theorem implies that
\[
\int_\Omega w (- \Delta \zeta + V \zeta) \d{x} = \int_{\partial\Omega} \frac{\partial \zeta}{\partial n} \d\nu \quad \text{for every \(\zeta \in C_0^\infty(\cl\Omega)\).}
\]
By uniqueness of distributional solutions, we have \(w = v\). Since \(H(v_k) \to H(v)\) almost everywhere on \(\Omega\), we deduce from Lebesgue's dominated convergence theorem that \(H(v_k) \to H(v)\) in \(L^1(\Omega)\). It then follows from \eqref{eq:W1p_estimate} that, taking a subsequence if necessary,
\[
\zeta_{H(v_k)} \to \zeta_{H(v)} \quad \text{almost everywhere on \(\Omega\).}
\]
Taking the limit in \eqref{eqNew1} as \(k \to \infty\), we obtain the conclusion when \(V\) is bounded.

In the general case where \(V\) is merely a nonnegative function in \(L^{1}_{\loc}(\Omega)\), we denote by \(w_k\) the solution of
\[
\begin{dcases}
\begin{aligned}
- \Delta w_k + V_k w_k &= 0 && \text{in \(\Omega\),} \\
w_k &= \nu && \text{on \(\partial\Omega\),}
\end{aligned}
\end{dcases}
\]
where \((V_k)\) is a sequence as in \Cref{lem:measure_limit}. By the first part of the proof, we have
\begin{equation}
\label{eqNew2}
w_k \ge \zeta_{H(w_k)} \quad \text{almost everywhere on \(\Omega\).}
\end{equation}
One then shows that \((w_k)\) is a non-increasing sequence of nonnegative functions such that \(w_k \to v\) almost everywhere on \(\Omega\) and in \(L^1(\Omega)\)\,; see~\cite{PonceWilmet:2020}*{Propositions~3.2~and~4.2}. One then deduces as above that \(H(w_k) \to H(v)\) in \(L^1(\Omega)\). Letting \(k \to \infty\) in \eqref{eqNew2}, the conclusion follows.
\end{proof}

\begin{proof}[Proof of \Cref{lem:zero-set}]
Let \(a \in \Gamma\). By \Cref{lem:comparison_principle}, we have
\[
P_a \ge \zeta_{H(P_a)} \quad \text{almost everywhere on \(\Omega\)\,;}
\]
whence
\[
\widehat{P_a} \ge \widehat{\zeta_{H(P_a)}} \quad \text{on \(\Omega\).}
\]
Since \(P_a \not\equiv 0\) and \(H > 0\) on \((0, +\infty)\), we have \(H(P_{a}) \ge 0\) and \(H(P_{a}) \not\equiv 0\). By Example~1.6 in \cite{OrsinaPonce:2020} we deduce that
\[
\widehat{P_a} \ge \widehat{\zeta_{H(P_a)}} > 0 \quad \text{on \(G\).}\qedhere
\]
\end{proof}

We now turn to the

\begin{proof}[Proof of \Cref{thm:hopf_lemma}]
\Cref{lem:nonzero_hopf_potential,lem:zero-set} imply that, for almost every \(a \in \Omega\),
\[
\widehat{P_a} > 0 \quad \text{on \(G\).}
\]
On the other hand, since \(\mu\) is nonnegative, Theorem~1.4 in \cite{OrsinaPonce:2020} ensures that \eqref{eq:dp} with datum \(\mu\) has a solution if and only if \(\mu(\Omega \setminus G) = 0\). Since \(u\) is a nontrivial solution with datum \(\mu\), we thus have that \(\mu \neq 0\) on \(G\). Therefore,
\[
\int_\Omega \widehat{P_a} \d\mu > 0 \quad \text{for almost every \(a \in \partial\Omega\).}
\]
The conclusion then follows from \Cref{thm:representation_formula}.
\end{proof}

\section*{Acknowledgements}

The first author (A.C.P.) was supported by the Fonds de la Recherche scientifique (F.R.S.--FNRS) under research grant  J.0020.18.


\begin{bibdiv}
\begin{biblist}

\bib{Ancona:1979}{article}{
   author={Ancona, Alano},
   title={Une propri\'et\'e d'invariance des ensembles absorbants par
   perturbation d'un op\'erateur elliptique},
   journal={Comm. Partial Differential Equations},
   volume={4},
   date={1979},
   number={4},
   pages={321--337},
}

\bib{BrezisMarcusPonce:2007}{article}{
   author={Brezis, Ha\"im},
   author={Marcus, Moshe},
   author={Ponce, Augusto C.},
   title={Nonlinear elliptic equations with measures revisited},
   conference={
      title={Mathematical aspects of nonlinear dispersive equations},
   },
   book={
      series={Ann. of Math. Stud.},
      volume={163},
      publisher={Princeton Univ. Press, Princeton, NJ},
   },
   date={2007},
   pages={55--109},
}

\bib{BrezisPonce:2003}{article}{
   author={Brezis, Ha\"im},
   author={Ponce, Augusto C.},
   title={Remarks on the strong maximum principle},
   journal={Differential Integral Equations},
   volume={16},
   date={2003},
   number={1},
   pages={1--12},
}

\bib{BrezisPonce:2008}{article}{
   author={Brezis, Ha\"im},
   author={Ponce, Augusto C.},
   title={Kato's inequality up to the boundary},
   journal={Commun. Contemp. Math.},
   volume={10},
   date={2008},
   number={6},
   pages={1217--1241},
}

\bib{DalMasoMosco:1986}{article}{
   author={Dal Maso, Gianni},
   author={Mosco, Umberto},
   title={Wiener criteria and energy decay for relaxed Dirichlet problems},
   journal={Arch. Rational Mech. Anal.},
   volume={95},
   date={1986},
   number={4},
   pages={345--387},
}

\bib{LittmanStampacchiaWeinberger:1963}{article}{
   author={Littman, Walter},
   author={Stampacchia, Guido},
   author={Weinberger, Hans F.},
   title={Regular points for elliptic equations with discontinuous
   coefficients},
   journal={Ann. Scuola Norm. Sup. Pisa (3)},
   volume={17},
   date={1963},
   pages={43--77},
}

\bib{MalusaOrsina:1996}{article}{
   author={Malusa, Annalisa},
   author={Orsina, Luigi},
   title={Existence and regularity results for relaxed Dirichlet problems with measure data},
   journal={Ann. Mat. Pura Appl. (4)},
   volume={170},
   date={1996},
   pages={57--87},
}

\bib{MarcusVeron:2014}{book}{
   author={Marcus, Moshe},
   author={V\'{e}ron, Laurent},
   title={Nonlinear second order elliptic equations involving measures},
   series={De Gruyter Series in Nonlinear Analysis and Applications},
   volume={21},
   publisher={De Gruyter, Berlin},
   date={2014},
}

\bib{OrsinaPonce:2018}{article}{
   author={Orsina, Luigi},
   author={Ponce, Augusto C.},
   title={Hopf potentials for the Schr\"{o}dinger operator},
   journal={Anal. PDE},
   volume={11},
   date={2018},
   number={8},
   pages={2015--2047},
}

\bib{OrsinaPonce:2020}{article}{
   author={Orsina, Luigi},
   author={Ponce, Augusto C.},
   title={On the nonexistence of Green's function and failure of the strong maximum principle},
   journal={J. Math. Pures Appl. (9)},
   volume={134},
   date={2020},
   pages={72--121},
}

\bib{Ponce:2016}{book}{
   author={Ponce, Augusto C.},
   title={Elliptic PDEs, measures and capacities. From the Poisson equations to nonlinear Thomas-Fermi problems},
   series={EMS Tracts in Mathematics},
   volume={23},
   publisher={European Mathematical Society (EMS), Z\"{u}rich},
   date={2016},
}

\bib{PonceWilmet:2020}{article}{
   author={Ponce, Augusto C.},
   author={Wilmet, Nicolas},
   title={The Hopf lemma for the Schr\"{o}dinger operator},
   journal={Adv. Nonlinear Stud.},
   volume={20},
   date={2020},
   number={2},
   pages={459--475},
}

\bib{Trudinger:1978}{article}{
   author={Trudinger, Neil S.},
   title={On the positivity of weak supersolutions of nonuniformly elliptic equations},
   journal={Bull. Austral. Math. Soc.},
   volume={19},
   date={1978},
   number={3},
   pages={321--324},
}

\end{biblist}
\end{bibdiv}
\end{document}